\documentclass[twoside,11pt,reqno]{amsart}

\usepackage{upref,amsxtra,amssymb,amscd}
\usepackage{varioref}
\usepackage{verbatim}
\usepackage{epsfig}
\usepackage{color}

\def\A{          \mathcal A}

\def\m{             \mathfrak m}

\def\IN{     \NN}    
\def\IR{     \RR}              
\newcommand{\NN}{{\mathbb N}}
\newcommand{\RR}{{\mathbb R}}
\newcommand{\TT}{{\mathbb T}}
\newcommand{\IT}{{\mathbb T}}

\newcommand{\ZZ}{{\mathbb Z}}

\newcommand{\DR}{{\mathcal DR}}

\newtheorem{theo}{\sc Theorem}[section]

\newtheorem{proposition}[theo]{\sc Proposition}
\newtheorem{lemma}[theo]{\sc Lemma}

\newtheorem{sublemma}[theo]{\sc Lemma}

\newtheorem{coro}[theo]{\sc Corollary}
\newtheorem{corollary}[theo]{\sc Corollary}

\theoremstyle{definition}

\newtheorem{defn}[theo]{\sc Definition}

\theoremstyle{remark}

\newtheorem{remark}[theo]{\sc Remark}

\numberwithin{equation}{section}
\begin{document}

\title{Remarks On Shrinking Target Properties}
\author{Jimmy Tseng}
\address{Jimmy Tseng, Department of Mathematics, The Ohio State University, Columbus, OH 43210}
\email{tseng@math.ohio-state.edu}

\begin{abstract}  
This paper defines and describes a few (related) notions of shrinking target property.  We show that simultaneous expanding circle maps have a certain shrinking target property, but that circle homeomorphisms and isometries of complete, separable metric spaces do not have any shrinking target property.  

\end{abstract}

\maketitle

\tableofcontents

\section{Introduction}\label{chapIntro}

Let $(M, \mu, G, T)$ denote a measure-preserving dynamical system on a metric space $M$ with a measure $\mu$, a countable semigroup $G$, and a homomorphism $T$ from $G$ into the semigroup of measure-preserving self-maps of $M$ (these maps are denoted by $T^g:M \rightarrow M$ for $g \in G$).  We also require the domain of $\mu$ to contain the Borel $\sigma$-algebra and $\mu(M)$ to be finite.  In this paper, we study a fundamental type of long-term behavior of such a system, namely discerning which sequences of balls do almost all orbits intersect infinitely often.  Such sequences can be used to characterize the system; and such characterizations are called shrinking target properties.  

Shrinking target properties (partly) arise from the well-known Borel-Cantelli lemma (see Lemma~1.1 of~\cite{CK} for the statement).  For our system $(M, \mu, G, T)$, that lemma (convergence case) says that almost no points are in a sequence of measurable sets infinitely often if the measures of these sets are summable.  If, alternatively, the measures of these sets are not summable, then one cannot make general (nontrivial) assertions for all dynamical systems.  However, it is useful (see, for example,~\cite{KM}) to identify those systems for which almost all points are in a sequence of measurable sets infinitely often (the measures of the sets must not be summable) and, preferably, not just for a few, but for a large family of such sequences; this desire to find such families gives rise to the shrinking target properties defined in Section~\ref{secMD}.  The study of such systems is, variously, called the study of dynamical Borel-Cantelli lemmas (for example,~\cite{KM} or~\cite{Ma}), quantitative Borel-Cantelli lemmas (\cite{Ph}), logarithm laws (for example,~\cite{Su},~\cite{KM}, or~\cite{AM}), hitting time (alternatively, waiting time; see, for example,~\cite{GP}), and shrinking target properties (for example,~\cite{Fa},~\cite{Ts},~\cite{BHKV},~\cite{T2},~\cite{ET}, or a survey of results,~\cite{At}).\footnote{The earliest known result is in~\cite{Ku}.  The term ``shrinking target,'' however, first appears in~\cite{HV}.}   

\subsection{Definitions}\label{secMD}  Let $(M, \mu, G, T)$ be as above and fix $s \geq 1$.  A sequence of measurable sets ${\lbrace  A_g \rbrace}_{g \in G }$ such that 
\begin{eqnarray} \label{div2} \sum_{g\in G}(\mu (A_g))^s =  \infty 
\end{eqnarray}
is called a \textit{Borel-Cantelli (BC) sequence} for $T$ if 
$\mu \left(\lim \sup T^{-g}A_g\right)=\mu(M)$.  (Note that, for $g \in G$, the notation $T^{-g}$ denotes the inverse of the map $T^g:M\rightarrow M$.)
A \textit{radius sequence} is a function $r:G \rightarrow \RR_{\geq 0}$.  Let $r_g := r(g)$, and denote the set of radius sequences by ${\mathcal R}(G)$.  An \textit{admissible set of radius sequences} ${\mathcal A}(G)$ is a nonempty subset of ${\mathcal R}(G)$.  

Let us state a general definition, which we informally call the \textit{admissible shrinking target property (ASTP)} and from which we will specify distinguished special cases:


\begin{defn}
Let $s\geq1$.  The dynamical system $(M, \mu, G, T)$ has the \textit{$(s,\A(G))$-shrinking target property ($(s,\A(G))$-STP)} if, for any $x \in M$ and any $r \in {\mathcal A}(G)$ such that $A_g := B(x, r_g)$ satisfies (\ref{div2}), $\{A_g \}$ is BC for $T$.
\end{defn}

Note that if a sequence satisfies (\ref{div2}) for $s >1$, then it also satisfies (\ref{div2}) for $s=1$.  Therefore, the convergence case of the Borel-Cantelli lemma does not apply to ASTP.  The divergence case of the lemma is also not useful since it requires independence. 

The notion of ASTP conveniently encapsulates three existing notions, which we now note.  Let $${\mathcal DR(\NN)} := \{r \in {\mathcal R(\NN)} \mid r_n \geq r_{n+1} \textrm{ for all } n \in \NN\}.$$   Then \begin{itemize} \item $(1,{\mathcal R(\NN)})$-STP is the \textit{shrinking target property (STP)}~\cite{Fa},  \item $(1,{\mathcal DR(\NN)})$-STP is the \textit{monotone shrinking target property (MSTP)}~\cite{Fa}, and \item $(s, \DR(\NN))$-STP is the \textit{$s$-exponent monotone shrinking target property ($s$MSTP)}~\cite{Ts}. \end{itemize}  Note that the relationship among these properties is given by the tower of implications $$\textrm{STP} \Rightarrow s\textrm{MSTP} \Rightarrow t\textrm{MSTP}$$ for $1 \leq s < t$~\cite{Ts} and further that, perhaps surprisingly, these properties are independent of mixing:  it is possible for a dynamical system $(M, \mu,\NN, T)$ to be mixing without having STP, MSTP~\cite{Fa}, or $s$MSTP for any $s\geq1$~\cite{GP}.  These properties are likely possessed by systems with hyperbolic behavior, and are possesed by, for example, any Anosov diffeomorphism with a smooth invariant measure (D. Dolgopyat,~\cite{Do}) or any linear expanding circle map (W. Philipp,~\cite{Ph}).  

Hyperbolic behavior, however, is not necessary for a system to have one of these properties.  For example, it has been shown in~\cite{Fa} that, while toral translations--which are elliptic systems and thus have no hyperbolic behavior--do not have STP (a result which also follows from Theorem~\ref{thmisometrySTP}), those translations whose translation vectors are in a certain set of full Hausdorff dimension do have MSTP~\cite{Ku} (see also~\cite{Fa}).  In dimension one, even more is known:  there are full-measure sets of rotation angles that correspond to circle rotations with $s$MSTP for all $s>1$~\cite{Ts}--this result is, moreover, sharp, and thus the $s$ can be thought of as a dynamical parameter for the family of circle rotations.   

\subsection{Statement of results}\label{subsecSoR}

Let us denote the torus by $\TT^n := \RR^n / \ZZ^n$ and the probability Haar measure on it by $\mu$.  In most of the cases of shrinking target properties considered in the aforementioned literature, the action is by the additive semigroup $\NN$ (with the natural ordering).  As a first step in generalizing to other actions, let us consider the semigroup $\NN^n$ with the following action:  for $k \in \IN^n$, define\footnote{It is easy to see that the set $\IN^n$ with binary operation $\circ:  \IN^n \times \IN^n \rightarrow \IN^n, k \circ l \mapsto (k_1 l_1, \cdots, k_n l_n)$ is a semigroup with identity $(1, \cdots, 1)$ and that $T^k$ is a $\IN^n$-action under $\circ$ and is $\mu$-preserving.}  \[T^k: \IT^n \rightarrow \IT^n; \alpha \mapsto (\alpha_1 k_1, \cdots, \alpha_n k_n ).\]  In Section~\ref{secMECMII}, we show

\begin{theo}
Let $n \geq 1$, $\mu$ be the probability Haar measure on $\TT^n$, and $T^k$ be as above.  Then there is a radius sequence \[r: \IN^n \rightarrow \IR_{\geq0}\] such that, for all real numbers $C >0$, the dynamical system $(\IT^n, \mu, \NN^n, T)$ has (1,$\{Cr\})$-STP.
\label{thmASTP}
\end{theo}

Next consider the multiplicative semigroup $\NN$ (with the natural ordering), which we denote by $(\NN, *)$ to distinguish it from the additive semigroup.  The proof of Theorem~\ref{thmASTP} simplifies to show

\begin{theo}
Let $\mu$ be the probability Haar measure on $\TT^1$ and $(\NN, *)$ be as above.  Consider the expanding circle maps $$ T^k:\TT^1 \rightarrow \TT^1; \alpha \mapsto k \alpha $$ where $k \in \NN$.  Then the dynamical system $(\TT^1, \mu, (\NN, *), T)$ has MSTP.
\label{corECMMSTP} 
\end{theo}

Let us now consider $\NN$ as the usual additive semigroup with the natural ordering.  Recall that toral translations do not have STP; our next result (proved in Section~\ref{secCH}) generalizes this in dimension one to circle homeomorphisms:

\begin{theo}
Let $\nu$ be a non-atomic Borel probability measure on $\TT^1$.  If \[f:\TT^1 \rightarrow \TT^1\] is a $\nu$-preserving homeomorphism, then the dynamical system $(\TT^1, \nu, \NN, f)$ does not have STP.
\label{thmCircleSTP}
\end{theo}

\begin{remark}  It is easy to see that the theorem is false if we remove the constraint that $\nu$ be non-atomic.  For example, take a rational rotation, and let  $supp(\nu)$ be, for example, the periodic orbit of $0$.  The only way for a sequence of balls to have divergent sum of measures is for a subsequence to always contain a point of the orbit of $0$.  Since the orbit of $0$ is finite, some point in this orbit must occur infinitely many times in this sequence of balls.   This argument, of course, generalizes to any dynamical system with a periodic orbit.

\end{remark}
\noindent Note that Theorem~\ref{thmCircleSTP} is false in higher dimensions, as a counterexample can be found in~\cite{Do}.

Using the well-known theory of rotation numbers of orientation-preserving circle homeomorphisms (see Chapter 7 of~\cite{BS} for a reference) and Lemma~\ref{lmirrrotnumnonatomic}, it immediately follows from the theorem that
\begin{coro}
Let $\nu$ be a Borel probability measure on $\TT^1$ and \[f:\TT^1 \rightarrow \TT^1\] be a $\nu$-preserving, orientation-preserving homeomorphism with irrational rotational number.  Then the dynamical system $(\TT^1, \nu, \NN, f)$ does not have STP.
\label{thmirrationalhomeonowSTP}
\end{coro}

\begin{remark}  Every orientation-reversing homeomorphism (and every orientation-preserving homeomorphism with rational rotation number) has a periodic point and hence, by the previous remark, can possibly have STP (with respect to an appropriate measure).  
\end{remark}

For irrational rotations,  the corollary reduces to the aforementioned result in~\cite{Fa} since the probability Haar measure on $\TT^1$ (restricted to the Borel $\sigma$-algebra) is the only invariant Borel probability measure.  However, the corollary (and the theorem) applies, for example, to the well-known Denjoy example (see Chapter 7 of~\cite{BS} for a reference).  

Finally, similar to Theorem~\ref{thmCircleSTP} is the following theorem for isometries of complete, separable metric spaces:

\begin{theo}
Let $M$ be a complete, separable metric space and $\nu$ be a non-atomic Borel probability measure on $M$.  If \[f:M \rightarrow M\] is a $\nu$-preserving isometry, then the dynamical system $(M, \nu, \NN, f)$ does not have STP.
\label{thmisometrySTP}
\end{theo}

\noindent Completeness and separability are used only to find a point of $supp(\nu)$ which is also a recurrent point for $f^{-1}$; otherwise, they are not necessary.

We show in Section~\ref{secCH} that a simplification of the proof of Theorem~\ref{thmCircleSTP} yields the proof of Theorem~\ref{thmisometrySTP}.  Let us study systems without STP first.

\subsection*{Acknowledgements} I would like to thank Dmitry Kleinbock for helpful discussions.

\section{Circle homeomorphisms and isometries of metric spaces}\label{secCH}

In this section, we study systems that do not have STP:  we show Theorems~\ref{thmCircleSTP} and~\ref{thmisometrySTP} and Corollary~\ref{thmirrationalhomeonowSTP}.  

\subsection{Notation}


Let $x, y \in \TT^1$, and let $\pi:\IR \rightarrow \TT^1$ be the canonical projection map.  Let $\bar{x}$ denote an element of $\pi^{-1}(x)$.  Then what is meant by an open ball or an open interval of $\TT^1$ is clear.  For example, $B(x, 1/2) = \TT^1 \backslash \{(x + 1/2) \textrm{ mod } 1\}.$ Moreover, there is a total ordering $\leq$ on $B(x, 1/2)$ derived from the natural total ordering on $\IR$, and $\leq$ is defined as follows:  for any $x, y \in B(x, 1/2)$, $x \leq y$ if $\bar{x} \leq \bar{y}$ whenever $\bar{x}$ and $\bar{y}$ lie in the same connected component of $\pi^{-1}(B(x,1/2)) \subset \IR$.  For convenience, let us define $\geq, <,$ and $>$ in the usual way from $\leq$.

The left interval of $B(x, 1/2)$ is denoted as \[B_-(x,1/2) := \{y \in B(x,1/2) \mid y < x \};\]  the right interval, \[B_+(x,1/2) := \{y \in B(x,1/2) \mid y > x \}.\]  Thus, $B(x,1/2) = B_-(x,1/2) \amalg B_+(x,1/2) \amalg \{x\}$.  The measure $\mu$ is the probability Haar measure on $\TT^1$, and $d$ denotes distance on $\TT^1$.  

\subsection{The nature of $supp(\nu)$}

There are three types of points $x$ (defined below) in $supp(\nu)$.\footnote{The measure $\nu$ is from the statement of Theorem~\ref{thmCircleSTP}.}  First note

\begin{lemma}
Let $\nu$ be a non-atomic, finite Borel measure.  Then $supp(\nu)$ is a perfect set.\footnote{Recall that a closed set of a metric space is \textit{perfect} if every point of it is a limit point.}
\label{lmsuppperfect}
\end{lemma}

\begin{proof}
The support is closed.  Let $x \in supp(\nu)$ be an isolated point.  Then there exists an open neighborhood of $x$, $U \subset \TT^1$, such that $U \cap supp(\nu) = \{x\}$.  Thus $\nu(\{x\}) = \nu(U)> 0$, a contradiction of non-atomic. \end{proof}

Define ${\mathcal S}(x)$ to be the set of all sequences in $supp(\nu) \backslash \{x\}$ converging to $x$.  Let $s \in {\mathcal S}(x)$.  Denote the $n$-th element of the sequence $s$ by $s_n$.  \begin{itemize} \item If, for all $s \in {\mathcal S}(x)$, there exists an $N(s) \in \IN$ such that, for all $n \geq N(s)$, $s_n \in B_-(x,1/2)$, then $x$ is \textit{isolated from the right}.  \item If, for all $s \in {\mathcal S}(x)$, there exists an $N(s) \in \IN$ such that, for all $n \geq N(s)$, $s_n \in B_+(x,1/2)$, then $x$ is \textit{isolated from the left}.

\item Otherwise, there exists a sequence $s \in {\mathcal S}(x)$ such that, for all $N \in \IN$, there exist $n, m \geq N$ such that $s_n \in B_+(x,1/2)$ and $s_m \in B_-(x,1/2)$; then $x$ is \textit{approached from both sides}. \end{itemize}

Now we can show

\begin{lemma}
Let $\nu$ be a non-zero, non-atomic, finite Borel measure on $\TT^1$.  Then \begin{enumerate}
 
\item  for every point $x \in supp(\nu)$ isolated from the left,  there exists a unique $y \in supp(\nu) \backslash \{x\}$ such that $\nu((y, x)) = 0$ or
\item  for every point $x \in supp(\nu)$ isolated from the right,  there exists a unique $y \in supp(\nu) \backslash \{x\}$ such that $\nu((x, y)) = 0$. \end{enumerate}


\label{lmmaxnullrad}
\end{lemma}

\begin{proof}
For the proof of (1), first note that $x$ cannot be the only point in $supp(\nu)$ since it would be an atom.  Assume not.  Then for every $y \in supp(\nu) \backslash \{x\}$, $\nu((y,x)) >0.$  Hence, $supp(\nu) \cap (y,x) \neq \phi$, a contradiction.  Hence, such a $y$ exists.  If it is not unique, let $z \in supp(\nu) \backslash \{x\}$ be another such point.  Then either $z \in (y, x)$ or $y \in (z,x)$, a contradiction.


The proof for (2) is analogous.  \end{proof}

For every point $x \in supp(\nu)$ isolated from the left, define \[s_x := \mu((y,x)),\] and for every point $x \in supp(\nu)$ isolated from the right, define \[s_x := \mu((x,y)).\]

The next lemma describes the nature of open intervals near a point of $supp(\nu)$ on the side that other points of $supp(\nu)$ are approaching.

\begin{lemma}
Let $\nu$ be a non-zero, non-atomic, finite Borel measure on $\TT^1$. \begin{enumerate} \item Let $x \in supp(\nu)$ and $a \in B_-(x, 1/2)$.  If $x$ is not isolated from the left, then there exists $\delta > 0$ (depending on $a$ and $x$) such that, if $b \in B_-(x, 1/2)$ and $d(b,x) < \delta$, then $supp(\nu) \cap (a,b) \neq \phi$. 
\item Let $x \in supp(\nu)$ and $a \in B_+(x, 1/2)$. If $x$ is not isolated from the right, then there exists $\delta > 0$ (depending on $a$ and $x$) such that, if $b \in B_+(x, 1/2)$ and $d(b,x) < \delta$, then $supp(\nu) \cap (b,a) \neq \phi$. 
\end{enumerate}
\label{lmcloseintervalshavemeasure}
\end{lemma}

\begin{proof}
We prove (1); the proof of the other assertion is analogous.  Assume not.  Then for every $\delta > 0$, there exists $b \in B_-(x, 1/2)$ such that $d(b,x) < \delta$ and $supp(\nu) \cap (a,b) = \phi$.  Then for all intervals \[(a, x-1/n) \cap supp(\nu) = \phi.\]  Hence, \[\cup_{n \geq N} (a, x-1/n) \cap supp(\nu) = \phi\] for some large $N \in \IN$.  Thus, \[(a,x) \cap supp(\nu) = \phi.\]  Since $x$ is not isolated from the left, there exists an element $y \in B_-(x,1/2) \cap supp(\nu)$ such that $y \in B(x,d(x,a))$.  Hence, $y \in (a,x)$, a contradiction. \end{proof}




\subsection{Proof for circle homeomorphisms}\label{subsecPCSTPTheorem}

Let $S$ be an infinite subset of $\IN$ which inherits the usual total ordering on $\IN$ (denoted by $\leq$).  A \textit{shrinking radius sequence} is a function $r:S \rightarrow \IR_{\geq 0}$ such that, for all $\varepsilon >0$, there exists an $N \in S$ such that, for all $n \in S \textrm{ and } n \geq N$, we have $r(n) < \varepsilon$.  Denote $r_s := r(s)$, and denote the set of shrinking radius sequences by ${\mathcal SR}(S)$.

Also, we need some standard notation.  Let $X$ be a set and $f:X \rightarrow X$ be a discrete-time dynamical system.  Let $x \in X$.  The \textit{$\omega$-limit set} of $x$ is \[\omega(x) := \bigcap_{n \in \NN} \overline{\bigcup_{i \geq n} f^i(x)}.\]  The point $x$ is called \textit{(positively) recurrent} if $x \in \omega(x)$.  The set of such recurrent points is denoted ${\mathcal R}(f)$ (see Chapter 2 of~\cite{BS} for a reference).

One needs to constrain the map in some way in order to not have STP:

\begin{lemma}
Any homeomorphism of $\TT^1$ onto itself takes open balls to open balls.
\label{lmhomeoballstoballs}
\end{lemma}

\begin{proof}
Let $f$ be the homeomorphism and $B$ be an open ball; B is open and connected.  Then $f^{-1}(B)$ is open and connected, hence a ball.  \end{proof}

Using this lemma, we can show the key step used in the proof of Theorem~\ref{thmCircleSTP}:

\begin{proposition}
Let $\nu$ be a non-zero, non-atomic, finite Borel measure on $\TT^1$ and \[f:\TT^1 \rightarrow \TT^1\] be a $\nu$-preserving homeomorphism.  Let $x \in supp(\nu)$.  Then there exists a sequence $\{n_k\} \subset \IN$ such that, for all $r \in {\mathcal SR}(\{n_k\})$, \[\nu(\cap_{l=1}^\infty \cup_{k = l}^\infty f^{-n_k}B(x,r_{n_k})) = 0.\]
\label{lmsubseqzerolimsup}
\end{proposition}


\subsubsection{Proof of Proposition~\ref{lmsubseqzerolimsup}}
Since $\TT^1$ is a complete, separable metric space, $\nu$ is a non-zero, finite Borel measure, and $f^{-1}$ is continuous and $\nu$-preserving, we have $supp(\nu) \subset {\mathcal R}(f^{-1})$.   Hence, there exists $\{n_k\}$ such that $f^{-n_k}(x) \rightarrow x$ as $k \rightarrow \infty$.  Fix this sequence $\{n_k\}$.


Let $r:\NN \rightarrow \RR_{\geq 0}$ be any radius sequence for the moment.  
By Lemma~\ref{lmhomeoballstoballs}, denote $f^{-n}B(x,r_n) = (a_n, b_n)$.  Also, define \[B^-_n = (a_n, b_n) \cap B_-(x, 1/2)\] and, likewise, \[B^+_n = (a_n, b_n) \cap B_+(x, 1/2).\]  Since $f$ is $\nu$-preserving, it follows that $\nu((a_n, b_n)) = \nu(B(x,r_n))$.

Now let  $r \in {\mathcal SR}(\{n_k\})$, which we consider for the remainder of the proof.  Thus, \begin{equation} \lim_{k \rightarrow \infty} \nu((a_{n_k}, b_{n_k})) = \lim_{k \rightarrow \infty} \nu(B(x,r_{n_k})) = \nu(\cap_{k=1}^\infty B(x, r_{n_k})) = \nu(\{x\}) = 0. \label{eqnvanmeas} \end{equation}

The following two (related) lemmas hold:

\begin{sublemma} If 
\begin{enumerate}
\item the point $x$ is not isolated from the left, then $\lim_{k \rightarrow \infty} \mu(B^-_{n_k})=0$.
\item the point $x$ is not isolated from the right, then $\lim_{k \rightarrow \infty} \mu(B^+_{n_k})=0$.
\end{enumerate}
\label{slmxnotleftisolated}
\end{sublemma}

\begin{proof}
We prove (1); the proof of the other assertion is analogous.  Assume not.  Denote $\limsup_{k \rightarrow \infty} \mu(B^-_{n_k})$ by $4r_0$.  Then $4r_0 > 0$.  Since $\mu$ is bounded between $0$ and $1$, there exists a convergent subsequence $\m \subset \{n_k\}$ such that, for $m \in \m$, the subsequential limit $\lim_{m \rightarrow \infty} \mu(B^-_m) > 2r_0$.  For $a:=x - r_0$ and $x$ the given point in $supp(\nu)$, take $\delta$ given by Lemma~\ref{lmcloseintervalshavemeasure}.

Fix $\varepsilon >0$ such that $0 < \varepsilon < min(\delta/2, r_0/2)$.  Then there exists an $M \in \m$ such that, for all $m \in \m$ whenever $m \geq M$, we have $d(f^{-m}(x),x) < \varepsilon$ and $\mu(B^-_m) > r_0$.

If $b_m > x$, then $x \in (a_m, b_m)$.  Thus, $\mu(a_m,x) > r_0$, and hence $a_m < x - r_0$.  Therefore, \[(x-r_0, x-\varepsilon) \subset (a_m, b_m).\]

If $b_m \leq x$, then $a_m < b_m - r_0 \leq x - r_0$.  Because $f^{-m}(x) \in (a_m, b_m)$, we have $f^{-m}(x) < b_m$.  Since $d(f^{-m}(x),x) < \varepsilon$, it follows that $x - \varepsilon < f^{-m}(x) <b_m$. Since $0 < \varepsilon < min(\delta/2, r_0/2)$, we have $x - r_0 < x- \varepsilon$.  Thus, again, \[(x-r_0, x-\varepsilon) \subset (a_m, b_m).\]

Since $d(x - \varepsilon, x) = \varepsilon < \delta$, Lemma~\ref{lmcloseintervalshavemeasure} implies that $ \nu(a_m, b_m) \geq \nu((x-r_0, x-\varepsilon)) > 0$ for all $m \in \m$ whenever $m \geq M$.  Hence, for $m \in \m$, we have the subsequential limit \[\lim_{m \rightarrow \infty} \nu((a_m, b_m)) \geq \nu((x-r_0, x-\varepsilon)) > 0,\] contradicting (\ref{eqnvanmeas}).\end{proof}

\begin{sublemma}
If \begin{enumerate} \item the point $x$ is isolated from the right, then $\limsup_{k \rightarrow \infty} \mu(B^+_{n_k}) \leq s_x$.
\item the point $x$ is isolated from the left, then $\limsup_{k \rightarrow \infty} \mu(B^-_{n_k}) \leq s_x$.\end{enumerate}
\label{slmxrightisolated}
\end{sublemma}

\begin{proof}
We prove (1); the proof of the other assertion is analogous.  Assume not.  Denote $\limsup_{k \rightarrow \infty} \mu(B^+_{n_k})$ by $s_0$.  Then $s_0 > s_x$.  Denote $s_1 = \frac{s_0 + s_x} 2$.  Since $\mu$ is bounded between $0$ and $1$, there exists a convergent subsequence $\m \subset \{n_k\}$ such that, for $m \in \m$, the subsequential limit $\lim_{m \rightarrow \infty} \mu(B^+_m) > s_1$.  Hence, there exists an $M \in \m$ such that, for all $m \in \m$ whenever $m \geq M$, we have $\mu(B^+_m) > s_1$ and $d(f^{-m}(x),x) < s_x$.

By Lemma~\ref{lmmaxnullrad}, set $y$ to be the unique point of $supp(\nu)$ such that $\mu((x,y)) = s_x$.  Since $x$ is isolated from the right, it follows that $f^{-m}(x) \in B_-(x, 1/2) \amalg \{x\}$.  Since $a_m < f^{-m}(x) \leq x$, we have $B^+_m \supset (x, y + s_1 - s_x)$, which is an open set containing a point of $supp(\nu)$, namely $y$.  Thus, for $m \in \m$, we have the subsequential limit \[\lim_{m \rightarrow \infty} \nu((a_m, b_m)) \geq \nu((x, y + s_1 - s_x)) > 0,\] contradicting (\ref{eqnvanmeas}).  \end{proof}



Let us now consider three different possibilities for approaching $x$. \medskip

\noindent \textit{Case 1:  The point $x$ is approached from both sides.} \medskip

Denote $\mu((a_{n_k}, b_{n_k}))$ by $l_{n_k}$.  Then, $\lim_{k \rightarrow \infty} l_{n_k}= 0$.  Choose any $\varepsilon > 0$.  Then there exists a $K \in \IN$ such that, for all $k \geq K$, we have $d(f^{-n_k}(x), x) < \varepsilon$ and $l_{n_k} < \varepsilon$.  Since $f^{-n_k}(x) \in (a_{n_k}, b_{n_k})$, it follows that $(a_{n_k}, b_{n_k}) \subset B(x, 2 \varepsilon)$ for all $k \geq K$.  Thus, \[\cap_{l=K}^\infty \cup_{k = l}^\infty f^{-n_k}B(x,r_{n_k}) \subset B(x, 2 \varepsilon).\]  But, $\varepsilon$ is arbitrary; hence, \[\cap_{l=K}^\infty \cup_{k = l}^\infty f^{-n_k}B(x,r_{n_k}) \subset \{x\},\] and the result follows. \medskip

\noindent \textit{Case 2:  The point $x$ is isolated from the right.} \medskip

Then $x$ is not isolated from the left.  Choose $\varepsilon \in (0, s_x)$.  Then there exists a $K \in \IN$ such that, for all $k \geq K$, we have $d(f^{-n_k}(x), x) < \varepsilon$, $\mu(B^-_{n_k}) < \varepsilon$, and $\mu(B^+_{n_k}) < s_x+ \varepsilon$.  Thus, $a_{n_k} < f^{-n_k}(x) \leq x$ (since $f^{-n_k}(x) \in supp(\nu)$).  Hence, \[(a_{n_k}, b_{n_k}) \subset (x - 2\varepsilon, x+s_x + 2\varepsilon) = (x - 2\varepsilon, y + 2\varepsilon)\] for all $k \geq K$.  Thus, \[\cap_{l=K}^\infty \cup_{k = l}^\infty f^{-n_k}B(x,r_{n_k}) \subset (x - 2\varepsilon, y + 2\varepsilon).\]  But, as $(x, y)$ is not in $supp(\nu)$, it follows that \[\cap_{l=K}^\infty \cup_{k = l}^\infty f^{-n_k}B(x,r_{n_k}) \cap supp(\nu) \subset B(x,2\varepsilon) \cup B(y, 2\varepsilon).\] But, $\varepsilon$ can be arbitrarily small; hence, \[\cap_{l=K}^\infty \cup_{k = l}^\infty f^{-n_k}B(x,r_{n_k}) \cap supp(\nu) \subset \{x\} \cup \{y\},\] and the result follows. \medskip

\noindent \textit{Case 3:  The point $x$ is isolated from the left.} \medskip

This case is proved in the analogous way to the previous case.  This completes the proof of Proposition~\ref{lmsubseqzerolimsup}. 

\subsubsection{Finishing the proof for circle homeomorphisms}

The next lemma is used to construct a sequence of open balls whose sum of measures diverges.

\begin{lemma}
Let $\nu$ be a non-zero, finite Borel measure on a metric space $M$. \begin{enumerate}

\item Then there exists an $N \in \IN$ such that, for all $n \geq N$, there exists an $r \in \IR_{\geq 0}$ such that we have $\nu(B(x, r)) \geq 1/n$.  

\item Let $t_n:= \inf\{r \geq 0 \mid \nu(B(x, r)) \geq 1/n\}$.  If $x \in supp(\nu)$, then $\lim_{n \rightarrow \infty} t_n = 0$. 
\end{enumerate}
\label{lmposmeas}
\end{lemma}

\begin{proof}

For the proof of (1), choose $N$ such that $N^{-1} < \nu(M)$.  Since $M = \cup_{m=1}^\infty B(z, m)$ for any $z \in M$, continuity of $\nu$ from below implies that there exists an $m_0$ large enough so that $\nu(B(z, m_0)) > N^{-1}$. 

For the proof of (2), note first that $\{t_n\}$ is monotonically decreasing since if $t_n < t_{n+1}$, $\nu(B(x,t_n+\frac{t_{n+1} -t_n} {2})) < (n+1)^{-1}$, but also $\nu(B(x,t_n +\frac{t_{n+1} -t_n} {2}))\geq n^{-1}$, a contradiction.

Let $x \in supp(\nu)$.  Denote $\lim_{n \rightarrow \infty} t_n$ by $t_0$.  Assume $t_0 > 0$.  Then $\nu(B(x, t_0/2)) < 1/n$ for all $n \geq N$.  Thus, $\nu(B(x,t_0/2)) = 0$, a contradiction.  \end{proof}    

We can now show:

\begin{proof}[Proof of Theorem~\ref{thmCircleSTP}]  Let $x \in supp(\nu)$.  By Proposition~\ref{lmsubseqzerolimsup}, there exists a sequence $\{n_k\} \subset \IN$ such that, for all $s \in {\mathcal SR}(\{n_k\})$, we have $\nu(\cap_{l=1}^\infty \cup_{k = l}^\infty f^{-n_k}B(x,s_{n_k})) = 0$.  

Choose $r_{n_k} := 2 t_k$ from Lemma~\ref{lmposmeas}.  Thus, $r \in {\mathcal SR}(\{n_k\})$.  Hence, it follows that \[ \nu(\cap_{l=1}^\infty \cup_{k = l}^\infty f^{-n_k}B(x,r_{n_k})) = 0. \] Also, by construction, we have $\sum_{k=N}^\infty \nu(B(x,r_{n_k})) \geq \sum_{k=N}^\infty 1/k =  \infty$.  

For $n \in \IN \backslash \{n_k\}_{k=N}^\infty$, let $r_n = 0$.  Hence, $B(x,r_n) = \emptyset$, and the sum and the limsup set remain the same.  Thus, the system does not have STP.  \end{proof}

Finally, to show Corollary~\ref{thmirrationalhomeonowSTP} from the theorem, we need the following:

\begin{lemma}
Let $\nu$ be a Borel probability measure on $\TT^1$ and $f:\TT^1 \rightarrow \TT^1$ be a $\nu$-preserving, orientation-preserving homeomorphism.  If $f$ has irrational rotation number, then $\nu$ is non-atomic.
\label{lmirrrotnumnonatomic}
\end{lemma}

\begin{proof}
Let $y \in supp(\nu)$.  If $\nu(\{y\}) > 0$, then, by $\nu$-preserving, every element in its orbit has the same positive measure.  Let ${\mathcal O}_f^+(y)$ denote the forward orbit of $y$ under $f$.  Since $\nu(\TT^1) < \infty$, $card({\mathcal O}_f^+(y)) < \infty$.  Hence, $y$ is a periodic point.  Thus, $\omega(y) = {\mathcal O}_T^+(y)$, contradicting Proposition~7.1.7 of~\cite{BS} which states that $\omega(y)$ is perfect. \end{proof}


\subsection{Proof for isometries}

Let us now note what simplifications of the proof of Theorem~\ref{thmCircleSTP} are necessary to prove Theorem~\ref{thmisometrySTP}.  

\begin{proof}[Proof of Theorem~\ref{thmisometrySTP}]
We need not consider the nature of $supp(\nu)$.  The two lemmas in the proof of Proposition~\ref{lmsubseqzerolimsup} can be replaced with  \[\lim_{k \rightarrow \infty} diam(f^{-n_k}(B(x,r_{n_k}))) = 0, \] which follows immediately because the map is an isometry.  The remainder of the proof of Proposition~\ref{lmsubseqzerolimsup} follows as in the case where $x$ is approached on both sides.  Now follow the rest of the proof of Theorem~\ref{thmCircleSTP}.
\end {proof}

\section{Simultaneous expanding circle maps}\label{secMECMII}


In this section, we study systems with a shrinking target property:  we show Theorems~\ref{thmASTP} and~\ref{corECMMSTP} (and Corollary~\ref{corASTP}).  To do so, we use the Khintchine--Groshev theorem according to Schmidt (Theorem 1 of ~\cite{Sch2}):

\begin{theo}
Consider the circle $\TT^1$ with the probability Haar measure $\mu$.  Let $n \geq 1$, and let $P_1(q), \cdots, P_n(q)$ be nonconstant polynomials with integral coefficients.  For each of the integers $j=1, \cdots, n$, let \[I_{j1} \supseteq I_{j2} \supseteq \cdots\] be a sequence of nested intervals in $\TT^1.$  Put $\psi(q) = \mu(I_{1q}) \cdots \mu(I_{nq})$ and $$\Psi(h) = \sum_{q=1}^h \psi(q).$$  Put $N(h; \alpha_1, \cdots, \alpha_n)$ for the number of integers $q, 1 \leq q \leq h,$ with $$\alpha_j P_j(q) \textrm{ mod } 1 \in I_{jq} \qquad (j=1, \cdots, n).$$ Let $\varepsilon > 0.$  Then $$ N(h; \alpha_1, \cdots, \alpha_n) = \Psi(h) + O(\Psi(h)^{1/2 + \varepsilon})$$ for almost every $n$-tuple of real numbers $\alpha_1, \cdots, \alpha_n.$
\label{thmKGS}
\end{theo}

\begin{proof}[Proof of Theorem~\ref{thmASTP}]
Let $P_1, \cdots, P_n$ be nonconstant polynomials in one variable with integral coefficients such that \[P_j(z) \rightarrow + \infty \textrm{ as } z \rightarrow +\infty\] for all $j = 1, \cdots ,n$.  Then there exists an $N_0 \in \IN$ such that, for all $z \geq N_0$ and all $j = 1, \cdots, n$, the polynomial $P_j(z)$ is an injection (since these polynomials are asymptotically strictly increasing) and $\geq 1$. Let $x$ be any point in $\IT^n.$

Define a radius sequence $r: \IN^n \rightarrow \IR_{\geq 0}$ as follows:
\[r(P_1(N_0), \cdots, P_n(N_0)) \geq r(P_1(N_0+1), \cdots, P_n(N_0+1)) \geq \cdots\]
and, for all $k \in \IN^n \backslash \{(P_1(q), \cdots,  P_n(q)) \mid q \in \IN \textrm{ and  }q \geq N_0\},$ let \[r(k) = 0.\]  Then $\{Cr\}$ is an admissible set of radius sequences.  Let us also denote $R_q := r(P_1(q), \cdots,  P_n(q)).$ 

Note that \begin{eqnarray*} A:=\sum_{k \in \IN^n} \mu(B(x, C r_k)) = \sum_{q \in \IN} \mu(B(x, C R_q)) = \sum_{q \in \IN} \prod_{j=1}^n \mu(B(x_j, C R_q)).\label{sumA} \end{eqnarray*}  Since the measures of balls are nonnegative, the order of summation is irrelevant.  If $A < \infty,$ then the system has $(1,\{Cr\})$-STP.

Otherwise, $A = \infty$; and we have \[\sum_{q \in \IN} \prod_{j=1}^n \mu(B(x_j, C R_q)) = \infty.\]  Since $I_{jq} := B(x_j,C R_q)$ are, by construction, nested, Theorem~\ref{thmKGS} asserts that, for $j=1, \cdots, n$ simultaneously, \[\alpha_j P_j(q) \textrm{ mod } 1 \in B(x_j, C R_q)\] for infinitely many $q \in \IN$ and almost all $\alpha.$  Hence, \[T^{(P_1(q), \cdots, P_n(q))} (\alpha) \in B(x, C R_q)\] for infinitely many $q \in \IN$ and almost all $\alpha \in \IT^n$.

Consequently, \[\mu(\{\alpha \in \IT^n \mid T^k(\alpha) \in B(x,Cr_k) \textrm{ for infinitely many } k \in \IN^n\}) = \mu(\IT^n),\] and the system $(\IT^n, \mu, \NN^n, T)$ has $(1,\{Cr\})$-STP. \end{proof}

Now we adapt the previous proof to show Theorem~\ref{corECMMSTP}:

\begin{proof}[Proof of Theorem~\ref{corECMMSTP}]
Set $n=1$ in the proof of Theorem~\ref{thmASTP}.  Then $\circ$ coincides with the usual multiplication on $\IN$.  Let $P_1 = q \in \ZZ[q].$  Then any weakly decreasing sequence of radii is allowed.  Hence, the proof of Theorem~\ref{thmASTP}, in this case, shows MSTP.\end{proof} 

Finally, note that Theorem~\ref{thmASTP} can be slightly strengthened as follows.  Two radius sequences $r, s: \IN^n \rightarrow \IR_{\geq 0}$ are equivalent if there exists real numbers $C_1, C_2 > 0$ such that \[C_1 s(k) \leq r(k) \leq C_2 s(k) \textrm{ for all but finitely many } k \in \IN^n.\]  Let $[r]$ denote the equivalence class of $r$.  Then the following is easily shown:

\begin{corollary}
Let $n \geq 1$.  Let $r$ be a radius sequence that satisfies Theorem~\ref{thmASTP}.  Then $[r]$ is an admissible set of radius sequences, and $(\IT^n, \mu, \NN^n, T))$ has $(1,[r])$-STP.
\label{corASTP}
\end{corollary}

\section{Conclusion}

In this paper, we have classified a few systems according to their shrinking target property or lack of.  Our results have raised a few questions.  Can Theorem~\ref{corECMMSTP} be generalized to higher dimensional tori and/or for nonlinear expanding maps?  Can the conclusion of Theorem~\ref{corECMMSTP} be strengthened to STP?  Also, we know that Theorem~\ref{thmCircleSTP} is false for higher-dimensional manifolds~\cite{Do}.  However, if we impose a restriction like requiring the preimages of open balls to be open balls, which is always true for self-homeomorphisms of $\TT^1$ and which is essential to the proof of our theorem, can we conclude that, for $M$, a complete, separable metric space, and $\nu$, a non-atomic Borel probability measure on $M$, if $f:M \rightarrow M$ is a $\nu$-preserving homeomorphism, then the dynamical system $(M, \nu, \NN, f)$ does not have STP?\footnote{Instead of requiring the preimages of open balls to be open balls, a possible alternative restriction is to require the self-map $f$ to be quasiconformal.}

\end{document}